\documentclass[12pt, reqno]{amsart} 
\usepackage{amsmath, amsthm, amscd, amsfonts, amssymb, graphicx, color}
\usepackage[bookmarksnumbered, colorlinks, plainpages]{hyperref}
\hypersetup{colorlinks=true,linkcolor=red, anchorcolor=green, citecolor=cyan, urlcolor=red, filecolor=magenta, pdftoolbar=true}

\newtheorem{theorem}{Theorem}[section]
\newtheorem{lemma}[theorem]{Lemma}

\newtheorem{cor}[theorem]{Corollary}
\theoremstyle{definition}

\theoremstyle{remark}
\newtheorem{remark}[theorem]{\bf{Remark}}
\numberwithin{equation}{section}
\begin{document}

\title [Sharper bounds for the numerical radius of $ \lowercase{n}\times \lowercase{n}$ operator matrices]    
{ {Sharper bounds for the numerical radius  of \\$ \lowercase{n}\times \lowercase{n}$ operator matrices }}

\author[P. Bhunia]{Pintu Bhunia}

\address {{Department of Mathematics, Indian Institute of Science, Bengaluru 560012, Karnataka, India}}
\email {\bf{pintubhunia5206@gmail.com; pintubhunia@iisc.ac.in }}


\thanks{The author would like to sincerely acknowledge Professor K. Paul for his valuable comments.
 The author also would like to thank SERB, Govt. of India for the financial support in the form of National Post Doctoral Fellowship (N-PDF, File No. PDF/2022/000325) under the mentorship of Professor Apoorva Khare}

\thanks{}


\date{}
\maketitle

\begin{abstract}
Let $A=\begin{bmatrix}
	A_{ij}
\end{bmatrix}$ be an $n\times n$ operator matrix, where each $A_{ij}$ is a bounded linear operator on a complex Hilbert space. Among other numerical radius bounds, we show that $w(A)\leq w(\hat{A})$, where $\hat{A}=\begin{bmatrix}
\hat{a}_{ij}
\end{bmatrix}$ is an  $n\times n$ complex matrix, with  
$$\hat{a}_{ij}= \begin{cases}
	w(A_{ii}) \text{ when $i=j$,} \\
	\left \| | A_{ij}|+ | A_{ji}^*| \right\|^{1/2} \left\|  | A_{ji}|+ | A_{ij}^*| \right\|^{1/2} \text{ when $i<j$,} \\
	0 \text{ when $i>j$} .
\end{cases}$$
This is a considerable improvement of the existing bound $w(A)\leq w(\tilde{A})$, where $\tilde{A}=\begin{bmatrix}
	\tilde{a}_{ij}
\end{bmatrix}$ is an  $n\times n$ complex matrix, with  
$$\tilde{a}_{ij}= \begin{cases}
	w(A_{ii}) \text{ when $i=j$}, \\
	\|A_{ij}\| \text{ when $i\neq j$}.
\end{cases}$$
Further, applying the bounds, we develop the numerical radius bounds for the product of two operators and the commutator of operators. Also, we develop an upper bound for the spectral radius of the sum of the product of $n$ pairs of operators, which improve the existing bound.

\end{abstract}

\section{\textbf{Introduction}}
\noindent Let $\mathcal{H}$ be a complex Hilbert space with inner product $\langle.,.\rangle$ and the corresponding norm $\|\cdot\|,$ and let $\mathcal{B}(\mathcal{H})$ denote the $C^*$-algebra of all bounded linear operators on $\mathcal{H}.$ If $\mathcal{H}$ is an $n$-dimensional Hilbert space, then $\mathcal{B}(\mathcal{H})$ is identified with $\mathcal{M}_n(\mathbb{C}),$ the set of all $n\times n$ complex matrices.
For $A\in \mathcal{B}(\mathcal{H})$, let $|A|=(A^*A)^{1/2}$ and $|A^*|=(AA^*)^{1/2},$ where $A^*$ is the adjoint of $A.$ Let $ r(A), \|A\|$ and $w(A)$ denote the spectral radius, the operator norm and the numerical radius of $A,$ respectively. The numerical radius of $A$ is defined as $$w(A)=\sup\left\{ |\langle Ax,x\rangle|: x\in \mathcal{H}, \|x\|=1  \right\}.$$
It is well known that the numerical radius, $w(\cdot): \mathcal{B}(\mathcal{H})\to \mathbb{R}$ defines a norm on $\mathcal{B}(\mathcal{H})$ and  satisfies the inequalities $\frac12 \|A\|\leq w(A)\leq \|A\|,$ for every $A\in \mathcal{B}(\mathcal{H}).$
Observe  that $w(A)=\|A\|$ if $A$ is normal and $w(A)=\frac12 \|A\|$ if $A^2=0.$ Analogous to the operator norm, the numerical radius satisfies the power inequality, i.e., $w(A^n) \leq w(A)^n$ for every positive integer $n.$ For further readings on the numerical radius, we refer to books \cite{book, book2} and for recent improvements of the numerical radius inequalities, the interested readers can see the articles \cite{Bhunia_ASM_2023,Bhu23_2,Bhu23,Bhunia_LAMA_2022,Bhunia_LAA_2021,Bhunia_AM_2021,Bhunia_RIM_2021,Bhunia_BSM_2021,Bhunia_LAMA_2021,Kittaneh_STD_2005,Kittaneh_2003}.

\noindent Let $A_{ij}\in \mathcal{B}(\mathcal{H})$ for all $i,j=1,2,\ldots,n.$ Then the $n\times n$ operator matrix, $\begin{bmatrix}
		A_{11}& A_{12}& \ldots& A_{1n}\\
	A_{21}& A_{22}& \ldots& A_{2n}\\
	\vdots&\vdots&  & \vdots\\
	A_{n1}& A_{n2}& \ldots& A_{nn}\\
\end{bmatrix}\in \mathcal{B}\left(\oplus_{i=1}^n\mathcal{H}\right).$ The operator matrices, a useful tool in studying Hilbert space operators, have been studied  over the years, see \cite{Halmos}. In 1995, Hou and Du \cite{Hou_1995} have proved that 
\begin{eqnarray}\label{p7}
	w\left(\begin{bmatrix}
		A_{11}& A_{12}& \ldots& A_{1n}\\
		A_{21}& A_{22}& \ldots& A_{2n}\\
		\vdots&\vdots&  & \vdots\\
		A_{n1}& A_{n2}& \ldots& A_{nn} 
	\end{bmatrix}  \right) \leq 	w\left(\begin{bmatrix}
	\|A_{11}\|& \|A_{12}\|& \ldots& \|A_{1n}\|\\
	\|A_{21}\|& \|A_{22}\|& \ldots& \|A_{2n}\|\\
	\vdots&\vdots&  & \vdots\\
	\|A_{n1}\|& \|A_{n2}\|& \ldots& \|A_{nn}\| 
\end{bmatrix}  \right).
\end{eqnarray}
Further, in 2015, Abu-Omar and Kittaneh \cite{Abu_LAA_2015} have developed a considerable refinement of the bound in \eqref{p7}, that is 
\begin{eqnarray}\label{p8}
	w\left(\begin{bmatrix}
		A_{11}& A_{12}& \ldots& A_{1n}\\
		A_{21}& A_{22}& \ldots& A_{2n}\\
		\vdots&\vdots&  & \vdots\\
		A_{n1}& A_{n2}& \ldots& A_{nn} 
	\end{bmatrix}  \right) \leq 	w\left(\begin{bmatrix}
		w(A_{11})& \|A_{12}\|& \ldots& \|A_{1n}\|\\
		\|A_{21}\|& w(A_{22})& \ldots& \|A_{2n}\|\\
		\vdots&\vdots&  & \vdots\\
		\|A_{n1}\|& \|A_{n2}\|& \ldots& w(A_{nn}) 
	\end{bmatrix}  \right).
\end{eqnarray}

 In this article, we develop a bound for the numerical radius of $n\times n$ operator matrices, which improve the bound in \eqref{p8}. Applying the bound, we obtain some bounds for the numerical radius of the product of two operators and the commutator of operators. Other results are also derived. Finally, we obtain an upper bound for the spectral radius of the sum of the product of $n$ pairs of operators, which improve the existing bound.

\section{\textbf{Main Results}}
We present an upper bound for the numerical radius of $n\times n$ operator matrices, which refines the bound in \eqref{p8}. For this purpose, we begin with the following well known lemma.

\begin{lemma}\label{lem1} \cite{Kit88}
	Let $A\in \mathcal{B}(\mathcal{H})$, and  let $x,y\in \mathcal{H}.$
	If $f,g :[0,\infty)\to [0,\infty)$ are continuous functions, satisfy $f(\lambda)g(\lambda)=\lambda$, for all $\lambda \in [0,\infty)$, then
	$$ |\langle Ax,y\rangle| \leq \|f(|A|)x\| \|g(|A^*|)y\|.$$ 
\end{lemma}

Now we prove the following theorem.

 \begin{theorem}\label{th1}
 	Let $A_{ij}\in \mathcal{B}(\mathcal{H})$ for all $i,j=1,2,\ldots,n.$ If $f,g :[0,\infty)\to [0,\infty)$ are continuous functions, satisfy $f(\lambda)g(\lambda)=\lambda$, for all $\lambda \in [0,\infty)$, then
 	$$  w\left(  
 	\begin{bmatrix}
 		A_{11}& A_{12}& \ldots& A_{1n}\\
 		A_{21}& A_{22}& \ldots& A_{2n}\\
 		\vdots&\vdots&  & \vdots\\
 		A_{n1}& A_{n2}& \ldots& A_{nn}
 	\end{bmatrix}
 	\right) \leq w\left(\begin{bmatrix}
 		w(A_{11})& a_{12}& \ldots& a_{1n}\\
 		0& w(A_{22})& \ldots& a_{2n}\\
 		\vdots&\vdots&  & \vdots\\
 		0& 0& \ldots& w(A_{nn})
 	\end{bmatrix}
 	\right), $$
 	where 	$a_{ij}=\left \| f^2(|A_{ij}|)+g^2(|A_{ji}^*|)\right\|^{1/2} \left \| f^2(|A_{ji}|)+g^2(|A_{ij}^*|)\right\|^{1/2}$.
 	
 \end{theorem}
\begin{proof}
	Take ${A}=\begin{bmatrix}
		A_{11}& A_{12}& \ldots& A_{1n}\\
		A_{21}& A_{22}& \ldots& A_{2n}\\
		\vdots&\vdots&  & \vdots\\
		A_{n1}& A_{n2}& \ldots& A_{nn}\\
	\end{bmatrix}$ and  $x=  \begin{bmatrix}
	x_1\\
	x_2\\
	\vdots\\
	x_n
\end{bmatrix}\in \oplus_{i=1}^{n} \mathcal{H}$ with $\|x\|=1$  (i.e., $\|x_1\|^2+\|x_2\|^2+\ldots+\|x_n\|^2=1).$
Now,
\begin{eqnarray}\label{p1}
	|\langle Ax,x\rangle|& = & \left| \sum_{i,j=1}^{n}  \langle A_{ij}x_j,x_i\rangle \right|\notag\\
	& \leq &  \sum_{i,j=1}^{n}  \left| \langle A_{ij}x_j,x_i\rangle \right|\notag\\
	&=& \sum_{i=1}^{n}  \left| \langle A_{ii}x_i,x_i\rangle \right|+\sum_{\underset{i\neq j}{i,j=1}}^{n}  \left| \langle A_{ij}x_j,x_i\rangle \right|\notag\\
	&=& \sum_{i=1}^{n}  \left| \langle A_{ii}x_i,x_i\rangle \right|+\sum_ {{\underset{i< j}{i,j=1}}}^{n}  \left( \left| \langle A_{ij}x_j,x_i\rangle \right|+\left| \langle A_{ji}x_i,x_j\rangle \right| \right)\notag\\
		&\leq & \sum_{i=1}^{n}  \left| \langle A_{ii}x_i,x_i\rangle \right| \notag \\
		 &+& \sum_ {{\underset{i< j}{i,j=1}}}^{n}  \left( \left\| f(| A_{ij}|)x_j \right\|   \left\|g(| A_{ij}^*|)x_i \right\|  +\left\| f(| A_{ji}|)x_i \right\|   \left\|g(| A_{ji}^*|)x_j \right\| \right), \notag\\ 
\end{eqnarray}
where the last inequality follows by using Lemma \ref{lem1}. Now, by Cauchy-Schwarz inequality, we get  
\begin{eqnarray*}
	&& \left\| f(| A_{ij}|)x_j \right\|   \left\|g(| A_{ij}^*|)x_i \right\|  +\left\| f(| A_{ji}|)x_i \right\|   \left\|g(| A_{ji}^*|)x_j \right\|\\
	 &&\leq  \left(	\left\| f(| A_{ij}|)x_j \right\|^2+   \left\|g(| A_{ji}^*|)x_j \right\|^2 \right)^{1/2} \left(\left\| f(| A_{ji}|)x_i \right\|^2+ \left\|g(| A_{ij}^*|)x_i \right\|^2 \right)^{1/2}\\
	 &&=	\left\langle \left(  f^2 (| A_{ij}|)+ g^2(| A_{ji}^*|)\right) x_j, x_j \right\rangle ^{1/2} 
	 \left\langle \left(  f^2 (| A_{ji}|)+ g^2(| A_{ij}^*|)\right) x_i, x_i \right\rangle ^{1/2}.
\end{eqnarray*}
Therefore, from \eqref{p1}, we obtain that
\begin{eqnarray*}
	|\langle Ax,x\rangle| &\leq& \sum_{i=1}^{n} w(A_{ii})\|x_i\|^2\\ &+&\sum_ {{\underset{i< j}{i,j=1}}}^{n}  w^{1/2}\left(  f^2 (| A_{ij}|)+ g^2(| A_{ji}^*|)\right) w^{1/2} \left(  f^2 (| A_{ji}|)+ g^2(| A_{ij}^*|)\right) \|x_i\|\|x_j\|\\
	&=& \left \langle \hat{A} |x|, |x| \right \rangle,
\end{eqnarray*}
where $|x|=\begin{bmatrix}
	\| x_1\|\\
	 \|x_2\|\\
	  \vdots\\
	   \|x_n\|
\end{bmatrix}\in  \mathbb{C}^n$ is an unit vector and 
$\hat{A}=\begin{bmatrix} \hat{a}_{ij}
\end{bmatrix}$ is an $n\times n$ complex matrix,  with 
$$\hat{a}_{ij}= \begin{cases}
	w(A_{ii}) \text{ when $i=j,$} \\
	w^{1/2}\left(  f^2 (| A_{ij}|)+ g^2(| A_{ji}^*|)\right) w^{1/2} \left(  f^2 (| A_{ji}|)+ g^2(| A_{ij}^*|)\right) \text{ when $i<j,$} \\
	0 \text{ when $i>j$} .
\end{cases}$$
Therefore, $	|\langle Ax,x\rangle| \leq w(\hat{A})$ for all $x\in \oplus_{i=1}^n \mathcal{H}$ with $\|x\|=1.$ This implies $w(A)\leq w(\hat{A}),$ as desired. 
\end{proof}

By considering $f(\lambda)=\lambda^{t}$ and $g(\lambda)=\lambda^{1-t}$ ($0\leq t\leq 1$) in Theorem \ref{th1}, we obtain the following corollary.
\begin{cor}\label{cor1}
	If $A_{ij}\in \mathcal{B}(\mathcal{H})$ for all $i,j=1,2,\ldots,n,$  then
	$$  w\left(  
	\begin{bmatrix}
		A_{11}& A_{12}& \ldots& A_{1n}\\
		A_{21}& A_{22}& \ldots& A_{2n}\\
		\vdots&\vdots&  & \vdots\\
		A_{n1}& A_{n2}& \ldots& A_{nn}\\
	\end{bmatrix}
	\right) \leq w\left(\begin{bmatrix}
		w(A_{11})& a_{12}& \ldots& a_{1n}\\
		0& w(A_{22})& \ldots& a_{2n}\\
		\vdots&\vdots&  & \vdots\\
		0& 0& \ldots& w(A_{nn})\\
	\end{bmatrix}
	\right), $$
	where  $a_{ij}=\left\| |A_{ij}|^{2t}+|A_{ji}^*|^{2(1-t)}\right\|^{1/2} \left\| |A_{ji}|^{2t}+|A_{ij}^*|^{2(1-t)}\right\|^{1/2}$, $0\leq t \leq 1$. 
	
\end{cor}

Also, considering $t= \frac12$ in Corollary \ref{cor1}, we obtain the following bound.
\begin{cor}\label{cor2}
	If $A_{ij}\in \mathcal{B}(\mathcal{H})$ for all $i,j=1,2,\ldots,n,$  then
	$$  w\left(  
	\begin{bmatrix}
		A_{11}& A_{12}& \ldots& A_{1n}\\
		A_{21}& A_{22}& \ldots& A_{2n}\\
		\vdots&\vdots&  & \vdots\\
		A_{n1}& A_{n2}& \ldots& A_{nn}\\
	\end{bmatrix}
	\right) \leq w\left(\begin{bmatrix}
		w(A_{11})& a_{12}& \ldots& a_{1n}\\
		0& w(A_{22})& \ldots& a_{2n}\\
		\vdots&\vdots&  & \vdots\\
		0& 0& \ldots& w(A_{nn})\\
	\end{bmatrix}
	\right), $$
	where $a_{ij}=\left\| |A_{ij}| +|A_{ji}^*| \right \|^{1/2} \left\| |A_{ji}|+|A_{ij}^*| \right\|^{1/2}$. 
	
\end{cor}

\begin{remark}
It is well known  that if $A=[a_{ij}]$ is an $n\times n$ complex matrix such that $a_{ij}\geq 0$ for all $i,j=1,2,\ldots,n$, then 
\begin{eqnarray}\label{p9}
	w(A)=w\left(\frac{A+A^*}{2}\right)=r\left(\frac{A+A^*}{2}\right),
\end{eqnarray}
 where $r(\cdot)$ denotes the spectral radius, see in \cite[p. 44]{Horn}. Using the argument in \eqref{p9}, we obtain that
$$w\left(\begin{bmatrix}
	w(A_{11})& a_{12}& \ldots& a_{1n}\\
	0& w(A_{22})& \ldots& a_{2n}\\
	\vdots&\vdots&  & \vdots\\
	0& 0& \ldots& w(A_{nn})\\
\end{bmatrix}
\right)=w\left(\begin{bmatrix}
	w(A_{11})& \frac12 a_{12}& \ldots& \frac12 a_{1n}\\
	\frac12 a_{12} & w(A_{22})& \ldots& \frac12 a_{2n}\\
	\vdots&\vdots&  & \vdots\\
	\frac12 a_{1n} & \frac12 a_{2n}& \ldots& w(A_{nn})\\
\end{bmatrix}
\right), $$
where $a_{ij}=\left\| |A_{ij}| +|A_{ji}^*| \right \|^{1/2} \left\| |A_{ji}|+|A_{ij}^*| \right\|^{1/2}$.
And also, $$w\left(\begin{bmatrix}
	w(A_{11})& \|A_{12}\|& \ldots& \|A_{1n}\|\\
	\|A_{21}\|& w(A_{22})& \ldots& \|A_{2n}\|\\
	\vdots&\vdots&  & \vdots\\
	\|A_{n1}\|& \|A_{n2}\|& \ldots& w(A_{nn})\\
\end{bmatrix}
\right)=w\left(\begin{bmatrix}
	w(A_{11})& \frac12 b_{12}& \ldots& \frac12 b_{1n}\\
	\frac12 b_{12} & w(A_{22})& \ldots& \frac12 b_{2n}\\
	\vdots&\vdots&  & \vdots\\
	\frac12 b_{1n} & \frac12 b_{2n}& \ldots& w(A_{nn})\\
\end{bmatrix}
\right), $$
where $b_{ij}=\|A_{ij}\|+\|A_{ji}\|$. 
Clearly, $$a_{ij}= \left \| |A_{ij}| +|A_{ji}^*| \right\|^{1/2} \left\| |A_{ji}|+|A_{ij}^*| \right \|^{1/2} \leq \|A_{ij}\|+\|A_{ji}\|=b_{ij}, \text{ for all $i,j.$} $$
The spectral radius monotonicity of matrices with non-negative entries (see in \cite[p. 491]{Horn2} and the equalities in \eqref{p9} imply that if $A=\begin{bmatrix}
	a_{ij}
\end{bmatrix}$ and $B=\begin{bmatrix}
b_{ij}
\end{bmatrix}$ are $n\times n$ complex matrices with $0\leq a_{ij}\leq b_{ij}$ for all $i,j=1,2,\ldots,n,$ then $w(A)\leq w(B).$
Therefore, 
$$w\left(\begin{bmatrix}
	w(A_{11})& a_{12}& \ldots& a_{1n}\\
	0& w(A_{22})& \ldots& a_{2n}\\
	\vdots&\vdots&  & \vdots\\
	0& 0& \ldots& w(A_{nn})\\
\end{bmatrix}
\right) \leq w\left(\begin{bmatrix}
	w(A_{11})& \|A_{12}\|& \ldots& \|A_{1n}\|\\
	\|A_{21}\|& w(A_{22})& \ldots& \|A_{2n}\|\\
	\vdots&\vdots&  & \vdots\\
	\|A_{n1}\|& \|A_{n2}\|& \ldots& w(A_{nn})\\
\end{bmatrix}
\right).$$
Thus, the bound in Corollary \ref{cor2} improves the bound in \eqref{p8}.
\end{remark}

Next, we develop  the numerical radius bounds of $2\times 2$ operator matrices.
Considering $n=2$ in Theorem \ref{th1}, we obtain that the following bound for the numerical radius of $2\times 2$ operator matrices.

 \begin{cor}\label{cor3}
	Let $A,B,C,D\in \mathcal{B}(\mathcal{H})$. If $f,g :[0,\infty)\to [0,\infty)$ are continuous functions, satisfy $f(\lambda)g(\lambda)=\lambda$, for all $\lambda \in [0,\infty)$, then
	$$  w\left(  
	\begin{bmatrix}
		A& B\\
		C& D
	\end{bmatrix}
	\right) \leq \frac{w(A)+w(D)+ \sqrt{\big (w(A)-w(D)\big)^2+a^2  }}{2} ,$$
	where $a =\left \| f^2(|B|)+g^2(|C^*|)\right\|^{1/2} \left \| f^2(|C|)+g^2(|B^*|)\right \|^{1/2}$. In particular, for $f(\lambda)=g(\lambda)=\lambda^{1/2},$ 
	\begin{eqnarray*}
		w\left(  
		\begin{bmatrix}
			A& B\\
			C& D
		\end{bmatrix}
		\right) \leq \frac{w(A)+w(D)+ \sqrt{\big (w(A)-w(D)\big)^2+ \left \| |B| +|C^*|\right\| \left \| |C| + |B^*|\right \| }}{2}.
	\end{eqnarray*}
\end{cor}

\begin{proof}
	Taking $n=2$, and $A_{11}=A, A_{12}=B, A_{21}=C$ and $A_{22}=D$ in Theorem \ref{th1}, we obtain that
	\begin{eqnarray*}
		w\left(  
		\begin{bmatrix}
			A& B\\
			C& D
		\end{bmatrix}
		\right) &\leq &  w\left(\begin{bmatrix}
			w(A)& a \\
			0& w(D)
		\end{bmatrix}
		\right)\\
		&=&  w\left(\begin{bmatrix}
			w(A)& \frac12 a \\
			\frac12 a& w(D)
		\end{bmatrix}
		\right) (\text{by using \eqref{p9}})\\
		&=& r\left(\begin{bmatrix}
			w(A)& \frac12 a \\
			\frac12 a& w(D)
		\end{bmatrix}
		\right).
	\end{eqnarray*} 
This completes the proof. 
\end{proof}

\begin{remark}
(i)	Clearly, $\left \| |B| +|C^*|\right\| \left \| |C| + |B^*|\right \|\leq (\|B\|+\|C\|)^2$. Therefore, the bound in Corollary \ref{cor3} (second bound) is stronger than the bound \cite[Theorem 2.1]{Paul_IJMMS_2012}, that is 
	\begin{eqnarray*}
		w\left(  
		\begin{bmatrix}
			A& B\\
			C& D
		\end{bmatrix}
		\right) \leq \frac{w(A)+w(D)+ \sqrt{\big (w(A)-w(D)\big)^2+ \left (\| B\right\| + \left \| C \right \|)^2 }}{2},
	\end{eqnarray*}
	(ii) Considering $A=D=0$ in Corollary \ref{cor3}, we obtain that if $B,C\in \mathcal{B}(\mathcal{H})$, then
	\begin{eqnarray}\label{p2}
		w\left(  
		\begin{bmatrix}
			0& B\\
			C& 0
		\end{bmatrix}
		\right) \leq \frac12 { \left \| |B| +|C^*|\right\|^{1/2} \left \| |C| + |B^*|\right \|^{1/2} }.
	\end{eqnarray}
This bound is also given in \cite[Theorem 4]{Abu_LAA_2015}.\\
(iii) Clearly, for every $A\in \mathcal{B}(\mathcal{H})$,  $w\left(\frac{A+A^*}{2}\right)\leq w(A)$ and $w\left(\frac{A-A^*}{2}\right)\leq w(A).$ Therefore, using this argument and the bound in \eqref{p2}, we obtain that if $B,C\in \mathcal{B}(\mathcal{H})$ are self-adjoint, then
\begin{eqnarray}\label{p3}
\frac12 \|B\pm C\|\leq 	w\left(  
	\begin{bmatrix}
		0& B\\
		C& 0
	\end{bmatrix}
	\right) \leq \frac12  \left \| |B| +|C|\right\| .
\end{eqnarray}
Moreover, if $B,C\in \mathcal{B}(\mathcal{H})$ are positive (semi-definite), then
\begin{eqnarray}\label{p4}
		w\left(  
	\begin{bmatrix}
		0& B\\
		C& 0
	\end{bmatrix}
	\right) = \frac12  \left \| B +C \right\|,
\end{eqnarray}
which is also given in \cite[Corollary 3]{Abu_LAA_2015}.
\end{remark}

Next, by applying the bound in Corollary \ref{cor1}, we obtain the following numerical radius bound for the sum of the product of two pairs of operators. 

\begin{theorem}\label{th2}
	Let $A,B,C,D\in \mathcal{B}(\mathcal{H}).$ Then
	\begin{eqnarray*}
		w(AB\pm CD) &\leq & \frac14 \Big( \left\| |A|^{2t}+ |B^*|^{2(1-t)} \right\| \left\| |B|^{2t}+ |A^*|^{2(1-t)} \right\|  \\
		&& +\left\| |C|^{2t}+ |D^*|^{2(1-t)} \right\| \left\| |D|^{2t}+ |C^*|^{2(1-t)} \right\|  \Big),
	\end{eqnarray*} 
	for all $t\in [0,1].$ In particular, for $t=\frac12$,
	\begin{eqnarray*}
		w(AB \pm CD) \leq \frac14 \left( \left\| |A|+ |B^*| \right\| \left\| |B|+ |A^*| \right\|   +\left\| |C|+ |D^*| \right\| \left\| |D|+ |C^*| \right\|  \right).
	\end{eqnarray*}
\end{theorem}

\begin{proof}
	Considering $n=3$, and $A_{11}=A_{22}=A_{33}=A_{23}=A_{32}=0$, $A_{12}=A$, $A_{13}=C$, $A_{21}=B$ and $A_{31}=D$ in Corollary \ref{cor1}, we obtain that
	\begin{eqnarray}\label{p5}
		w\left(\begin{bmatrix}
			0&A&C\\
			B&0&0\\
			D&0&0
		\end{bmatrix}\right) &\leq& w\left(\begin{bmatrix}
		0&a&c\\
		0&0&0\\
		0&0&0
	\end{bmatrix}\right),
	\end{eqnarray}
	where $$a=\left\| |A|^{2t}+ |B^*|^{2(1-t)} \right\|^{1/2} \left\| |B|^{2t}+ |A^*|^{2(1-t)} \right\|^{1/2}, $$ 
	$$c=\left\| |C|^{2t}+ |D^*|^{2(1-t)} \right\|^{1/2} \left\| |D|^{2t}+ |C^*|^{2(1-t)} \right\|^{1/2} .$$
	Now, we see that $$ \begin{bmatrix}
		0&A&C\\
		B&0&0\\
		D&0&0
	\end{bmatrix}^2=\begin{bmatrix}
	AB+CD&0&0\\
	0&BA&BC\\
	0&DA&DC
\end{bmatrix}.$$
Therefore,
\begin{eqnarray*}
	w(AB+CD) &\leq& w\left(\begin{bmatrix}
		AB+CD&0&0\\
		0&BA&BC\\
		0&DA&DC
	\end{bmatrix}\right)\\
&=& w\left(\begin{bmatrix}
	0&A&C\\
	B&0&0\\
	D&0&0
\end{bmatrix}^2\right)\\
&\leq& w\left(\begin{bmatrix}
	0&A&C\\
	B&0&0\\
	D&0&0
\end{bmatrix}\right)^2 \\
&\leq& w\left(\begin{bmatrix}
	0&a&c\\
	0&0&0\\
	0&0&0
\end{bmatrix}\right)^2 (\text{using the inequality \eqref{p5}})\\
&=& r\left(\begin{bmatrix}
	0&\frac12 a&\frac12 c\\
	\frac12 a&0&0\\
	\frac12 c&0&0
\end{bmatrix}\right)^2 (\text{by using \eqref{p9}})\\
&=& \frac14 \left( a^2+c^2\right).
\end{eqnarray*}
This completes the proof.
\end{proof}

Now, taking $C=B$ and $D=A$ in Theorem \ref{th2}, we obtain the following numerical radius bound for the commutator of operators.

\begin{cor}\label{cor5}
	Let $A,B\in \mathcal{B}(\mathcal{H}).$ Then
	\begin{eqnarray*}
		w(AB\pm BA) &\leq & \frac12 \left\| |A|^{2t}+ |B^*|^{2(1-t)} \right\| \left\| |B|^{2t}+ |A^*|^{2(1-t)} \right\|,
	\end{eqnarray*} 
	for all $t\in [0,1].$ In particular, for $t=\frac12$,
	\begin{eqnarray}\label{p11}
		w(AB \pm BA) \leq \frac12  \left\| |A|+ |B^*| \right\| \left\| |B|+ |A^*| \right\| .
	\end{eqnarray}
\end{cor}

	In \cite{Kittaneh_STD_2005}, Kittaneh  proved that for $A,B\in \mathcal{B}(\mathcal{H}),$ 	\begin{eqnarray}\label{p12}
			w(AB \pm BA) \leq \frac12  \left\| |A|^2+ |A^*|^2 + |B|^2+ |B^*|^2 \right\| .
		\end{eqnarray}
	With the help of numerical examples, we conclude that the bound \eqref{p11} is, in general, incomparable with the bound  \eqref{p12}.

Also, if we consider $C=D=0$ in Theorem \ref{th2}, then we obtain the following numerical radius bound for the product of two operators.
\begin{cor}\label{cor6}
	If $A,B\in \mathcal{B}(\mathcal{H})$, then
	\begin{eqnarray*}
		w(AB) &\leq & \frac14  \left\| |A|^{2t}+ |B^*|^{2(1-t)} \right\| \left\| |B|^{2t}+ |A^*|^{2(1-t)} \right\|.
	\end{eqnarray*} 
	for all $t\in [0,1].$ In particular, for $t=\frac12$,
	\begin{eqnarray}\label{p14}
		w(AB) \leq \frac14 \left\| |A|+ |B^*| \right\| \left\| |B|+ |A^*| \right\|.
	\end{eqnarray}
\end{cor}


Further, taking $D=A$, $B=U$ and $C=U^*$ ($U$ is an unitary operator) in Theorem \ref{th2}, we obtain the following corollary.

\begin{cor}\label{cor7}
	Let $A,U\in \mathcal{B}(\mathcal{H}),$ where $U$ is an unitary operator. Then
	\begin{eqnarray*}
		w(AU\pm U^*A) &\leq & \frac12 \left\| |A|^{2t}+ I \right\| \left\|  |A^*|^{2(1-t)} +I\right\| .
	\end{eqnarray*} 
	for all $t\in [0,1].$ In particular, for $t=\frac12$,
\begin{eqnarray*}
			w(AU\pm U^*A) &\leq & \frac12 \left\| |A|+ I \right\| \left\|  |A^*| +I\right\|.
		\end{eqnarray*} 
	Also, for $t=0$ and $t=1$,
		\begin{eqnarray*}
		w(AU\pm U^*A) &\leq &  \min \left\{ \left\| |A|^{2}+ I \right\|, \left\| |A^*|^{2}+ I \right\| \right\}.
	\end{eqnarray*} 
\end{cor}

Finally, by applying the bound obtained in Corollary \ref{cor2}, we develop an upper bound for the spectral radius of $\sum_{i=1}^n A_iB_i$, where $A_i,B_i\in \mathcal{B}(\mathcal{H}).$

\begin{theorem} \label{th3}
	Let $A_i,B_i\in \mathcal{B}(\mathcal{H})$ for all $i=1,2,\ldots,n.$ Then
	\begin{eqnarray*}
		r\left(\sum_{i=1}^n A_iB_i\right) &\leq&  w \left(\begin{bmatrix}
			w(B_1A_1)& a_{12}& \ldots & a_{1n} \\
			0& w(B_2A_2)& \ldots& a_{2n} \\
			\vdots&\vdots&  & \vdots\\
			0& 0& \ldots& w(B_nA_n)\\
		\end{bmatrix} \right)\\
		&\leq&\max_{1\leq i\leq n } \left\{ w(B_iA_i)+ \frac12 \sum_{j=i+1}^na_{ij}  +\frac12 \sum_{j=1}^{i-1}a_{ji}\right\},
	\end{eqnarray*} 
where $a_{ij}= \left \| |B_iA_j|+ |A_i^*B_j^*|\right\|^{1/2}  \left \| |B_jA_i|+ |A_j^*B_i^*|\right\|^{1/2}.$
\end{theorem}
\begin{proof}
	Take $A=\begin{bmatrix}
		A_1& A_{2}& \ldots & A_{n} \\
		0& 0& \ldots& 0 \\
		\vdots&\vdots&  & \vdots\\
		0& 0& \ldots& 0\\
	\end{bmatrix}$ and $B=\begin{bmatrix}
	B_1& 0& \ldots & 0 \\
	B_2& 0& \ldots& 0 \\
	\vdots&\vdots&  & \vdots\\
	B_n& 0& \ldots& 0\\
\end{bmatrix}.$ We have,
\begin{eqnarray}\label{p10}
	r\left(\sum_{i=1}^n A_iB_i\right) =r(AB)=r(BA)\leq w(BA).
\end{eqnarray}
Therefore, the desired first inequality follows from  \eqref{p10} together with the bound in Corollary \ref{cor2}. The desired second inequality follows from the fact that
\begin{eqnarray*}
&&	\left| \left \langle \begin{bmatrix}
		w(B_1A_1)& a_{12}& \ldots & a_{1n} \\
		0& w(B_2A_2)& \ldots& a_{2n} \\
		\vdots&\vdots&  & \vdots\\
		0& 0& \ldots& w(B_nA_n)\\
	\end{bmatrix}   \begin{bmatrix}
		x_1 \\
		x_2  \\
		\vdots\\
		x_n\\
	\end{bmatrix}, \begin{bmatrix}
		x_1 \\
		x_2  \\
		\vdots\\
		x_n\\
	\end{bmatrix} \right \rangle \right|\\
&&\leq \sum_{i=1}^n \left\{ w(B_iA_i)+ \frac12 \sum_{j=i+1}^na_{ij}  +\frac12 \sum_{j=1}^{i-1}a_{ji}\right\}|x_i|^2,
\end{eqnarray*}
for all $\begin{bmatrix}
	x_1 \\
	x_2  \\
	\vdots\\
	x_n\\
\end{bmatrix}\in \mathbb{C}^n$ with $\sum_{i=1}^n|x_i|^2=1.$
\end{proof}

\begin{remark}
Clearly,
 \begin{eqnarray*}
	 & & \max_{1\leq i\leq n } \left\{ w(B_iA_i)+ \frac12 \sum_{j=i+1}^na_{ij}  +\frac12 \sum_{j=1}^{i-1}a_{ji}\right\}\\
	&\leq & \max_{1\leq i\leq n } \left\{ w(B_iA_i)+ \frac12 \sum_{j=i+1}^n (\|B_iA_j\|+\|B_jA_i\|)  +\frac12 \sum_{j=1}^{i-1}(\|B_jA_i\|+\|B_iA_j\|)\right\}\\
	&=& \max_{1\leq i\leq n } \left\{ w(B_iA_i)+ \frac12 \sum_{\underset{j\neq i}{j=1}}^n (\|B_iA_j\|+\|B_jA_i\|) \right\}. 
	\end{eqnarray*} 
Therefore, the spectral radius bound in Theorem \ref{th3} refines the existing bound in \cite[Theorem 2.10]{Bhunia_LAA_2019}, which  is
$$ r\left(\sum_{i=1}^nA_iB_i\right) \leq \max_{1\leq i\leq n } \left\{ w(B_iA_i)+ \frac12 \sum_{\underset{j\neq i}{j=1}}^n (\|B_iA_j\|+\|B_jA_i\|) \right\}.$$
\end{remark}

\noindent {\bf{Declarations.}}
\noindent {\bf{Conflict of Interest.}} The author declare that there is no conflict of interest.

\bibliographystyle{amsplain}

\end{document}